\documentclass[10pt,reqno]{article}
\usepackage{amsmath}
\usepackage{amscd,amsthm,amsfonts,amsopn}
\usepackage{graphics}

\makeatletter
\def\blfootnote{\gdef\@thefnmark{}\@footnotetext}
\makeatother

\newtheorem{theorem}{Theorem}[section]

\newtheorem{proposition}{Proposition}[section]

\newtheorem{lemma}{Lemma}[section]

\theoremstyle{remark}

\DeclareMathOperator{\ebdiv}{div}

\DeclareMathOperator*{\esssup}{esssup}

\begin{document}
\allowdisplaybreaks
\title{A bound from below on the temperature for the Navier-Stokes-Fourier system.}
\author{Eric Baer\footnote{Department of Mathematics, Massachusetts Institute of Technology; ebaer@math.mit.edu}
\and Alexis Vasseur\footnote{Department of Mathematics, The University of Texas at Austin; vasseur@math.utexas.edu}}
\date{April 8, 2013\blfootnote{Acknowledgement: This work is part of the first author's doctoral dissertation at the University of Texas at Austin (2012).  The first author was supported by the National Science Foundation under Award No. DMS-1204557.  The second author was partially supported by the NSF.}}
\maketitle
\begin{abstract}
We give a uniform bound from below on the temperature for a variant of the compressible Navier-Stokes-Fourier system, under suitable hypotheses.  This system of equations forms a mathematical model of the motion of a compressible fluid subject to heat conduction.  Building upon the work of \cite{MelletVasseur}, we identify a class of weak solutions satisfying a localized form of the entropy inequality (adapted to measure the set where the temperature becomes small) and use a form of the De Giorgi argument for $L^\infty$ bounds of solutions to elliptic equations with bounded measurable coefficients.
\end{abstract}
\section{Introduction}
Let $\Omega\subset \mathbb{R}^3$ be a bounded domain with smooth boundary.  We consider weak solutions to a variant of the Navier-Stokes-Fourier system, in the presence of no external forces and subject to heat conduction driven by Fourier's law:
\begin{align}
\label{nsf_system}\left\lbrace \begin{array}{rl} \partial_t\rho+\ebdiv (\rho u)&=0\\
\partial_t (\rho u)+\ebdiv (\rho u\otimes u)+\nabla p&=\ebdiv \mathbb{S}\\
\partial_t (\rho s)+\ebdiv (\rho s u)+\ebdiv(\frac{-\kappa\nabla \theta}{\theta})&=\sigma%\\
\end{array}\right.
\end{align}
in $(0,T)\times\Omega$, with the initial and boundary conditions
\begin{align*}
\left\lbrace\begin{array}{ll}\rho(0,\cdot)=\rho_0,\quad (\rho u)(0,\cdot)=(\rho u)_0,\quad \theta(0,\cdot)=\theta_0,\\
u(t,\cdot)|_{\partial\Omega}=0,\quad \nabla \theta(t,\cdot)\cdot n(\cdot)|_{\partial\Omega}=0.\end{array}\right.
\end{align*}

This system of equations models the motion of a viscous, compressible, and heat-conducting fluid, where $\rho=\rho(t,x)$ denotes the density of the fluid, $u=u(t,x)$ denotes the velocity of the fluid, and $\theta=\theta(t,x)$ denotes the temperature of the fluid.  The quantity $p$ determines the internal pressure of the system, while $\mathbb{S}$, $\kappa$, $s$ and $\sigma$ denote the stress tensor, heat conduction coefficient, entropy, and entropy production rate, respectively -- our assumptions on the behavior of these quantities are determined by the particular constitutive relations of our model; for more details, see the discussion after the statement of Theorem $\ref{thm_temp1}$, as well as the complete specification in Section $2$.

Recently, Mellet et al. \cite{MelletVasseur} studied bounds from below on the temperature for a suitable class of weak solutions of a variant of (\ref{nsf_system}) when the pressure $p(\rho,\theta)$ is affine in the temperature variable, i.e.
\begin{align*}
p=p_e(\rho)+R\rho\theta,
\end{align*}
in which case the entropy equation (that is, the third equation in (\ref{nsf_system})) is replaced by 
\begin{align*}
\partial_t (\rho\theta)+\ebdiv (\rho \theta u)-\ebdiv(\kappa\nabla\theta)=2\mu|D(u)|^2+\lambda |\ebdiv u|^2-R\rho\theta\ebdiv u.
\end{align*}

In \cite{MelletVasseur}, the authors use an instance of the De Giorgi argument \cite{DeGiorgi} for boundedness and regularity of solutions to elliptic equations with bounded measurable coefficients to establish uniform (in space) bounds on the logarithm of the temperature, which in turn give uniform bounds on the temperature itself.  

The goal of the present work is to adapt the methods of \cite{MelletVasseur} to treat the system ($\ref{nsf_system}$), in the case that the pressure is no longer strictly affine in the temperature variable.  This change in assumption on the pressure corresponds to a somewhat more physically accurate model; in particular, the constitutive assumptions on the quantities driving heat conduction in the system can now be related to basic thermodynamical principles (see \cite{FeireislCMAP,FeireislNovotnyBook} for further discussion on this point).

Our main result is then the following:
\begin{theorem}
\label{thm_temp1}
Fix $T>0$ and $\Omega$ a bounded open set.  Suppose that $\mathbb{S}$, $\kappa$, $\sigma$ and the state relations $s$ and $p$ (which respectively represent the entropy and pressure relations of the system) satisfy the criteria established in Section $2$, and let $(\rho,u,\theta)$ be a weak solution to the Navier-Stokes-Fourier system $(\ref{nsf_system})$ satisfying 
$u\in L^2(0,T;H^1_0(\Omega))$, 
\begin{align}
\int_{\Omega} \rho_0\max\left\{\log\left(\frac{1}{\theta_0}\right),0\right\}dx<\infty.\label{eqAA3}
\end{align}
and $\rho\in L^\infty(0,T;L^\omega(\Omega))$ for some $\omega>3$, along with the local entropy inequality ($\ref{eq_admiss}$) for a.e. $0<t<T$ and a.e. $0<\mathfrak{s}<t$, as well as for a.e. $0<t<T$ with $\mathfrak{s}=0$.\footnote{We shall describe the significance and relevance of these restrictions on the class of weak solutions in the discussion below.}

Then for all $\tau\in (0,T]$, there exists $\eta_{\tau,T}>0$ such that 
\begin{align*}
\theta(t,x)\geq \eta_{\tau,T}.
\end{align*}
for a.e. $\tau<t<T$ and a.e. $x\in\Omega$.
\end{theorem}

Theorem $\ref{thm_temp1}$ states that for a particular class of weak solutions, the temperature is bounded away from zero uniformly in space.\footnote{The proof in fact gives a slightly stronger statement, since the argument does not require that the triple $(\rho,u,s)$ satisfy the full conditions of a weak solution for ($\ref{nsf_system}$).  In particular, the only properties used in the argument are: (i) the bounds $u\in L^2(H_0^1)$, ($\ref{eqAA3}$), $\rho \in L^\infty(L^\omega)$ with $\omega>3$, (ii) the local entropy inequality ($\ref{eq_admiss}$) and (iii) the conservation of total mass $\int_{\Omega} \rho(t,x)dx=\int_{\Omega} \rho_0(x)dx$ (which follows from ($\ref{eq-weak-continuity}$)).  We use the present statement to emphasize the connection with the system ($\ref{nsf_system}$).} We remark that the assumptions on the system appearing in Section $2$ are all physically motivated and are quite general.  In particular, the quantities $p=p(\rho,\theta)$ and $s=s(\rho,\theta)$ represent the internal pressure and entropy of the system, and their precise forms along with those of the viscous stress tensor $\mathbb{S}$, heat conduction coefficient $\kappa$ and entropy production rate $\sigma$ are determined by the particular properties of the fluid under study.  We refer the reader to \cite[Chapter $1$]{FeireislNovotnyBook} for a full discussion of the derivation and physical relevance of the Navier-Stokes-Fourier system ($\ref{nsf_system}$).  

On the other hand, the assumptions on $\rho$, $u$ and $\theta(0,\cdot)$ are more closely connected with our tools and techniques.  As we mentioned above, the authors in \cite{MelletVasseur} use a variant of the De Giorgi argument for $L^\infty$ bounds of solutions to elliptic equations with measurable coefficients to establish the desired $L^\infty$ control over the logarithm of the temperature (which corresponds in our setting to the entropy, i.e. the quantity $s(\rho,\theta)$).  Generally speaking, this technique is based upon the balance of two key pieces of information: 
\begin{enumerate}
\item[(a)] a localized form of an energy/entropy inequality (e.g. the local energy inequality satisfied by suitable weak solutions for the incompressible Navier-Stokes equations; in our case, this takes the form of the local entropy ineuqality $(\ref{eq_admiss})$), and 
\item[(b)] a nonlinear iteration argument driven by the Tchebyshev inequality.
\end{enumerate}

As is often the case (see, e.g. the discussion in \cite{CKN} for the case of incompressible Navier-Stokes), in order to obtain an appropriate form of the local entropy inequality we must restrict the class of weak solutions.  In \cite{MelletVasseur}, the authors work with the solutions constructed by Feireisl in \cite{FeireislBook}, which arise as limits of a somewhat involved approximation procedure.  This procedure in particular preserves an appropriate form of the entropy inequality at the last level of the approximation, which enables the authors to obtain the desired $L^\infty$ bounds uniformly in the approximation parameter.  

\vspace{0.2in}

In the present work, we base our notion of weak solution on the existence theory developed by Feireisl, Novotn\'y et al. (see \cite[Chapter $3$]{FeireislNovotnyBook}, as well as the works \cite{FeireislBook,FeireislIUMJ53_2004,FeireislContempMath371,FeireislCommPDE31_2006,FeireislNovotnyProcRSE135A_2005}).\footnote{Note that the model described in these works contains an additional radiative term when compared to the system ($\ref{nsf_system}$); at present, this additional term is a required component of the known existence theory.  We discuss this in more detail at the conclusion of this introduction.}  In this setting, the identification of an appropriate form of the local entropy inequality is somewhat more subtle, since the entropy $s(\rho,\theta)$ may now depend on both the density $\rho$ and the temperature $\theta$ in a nonlinear way.  In particular, recalling that these localized inequalities are typically obtained by multiplying the equation by an appropriate cutoff function, one observes that the (possibly nonlinear) interaction of $\rho$ and $\theta$ inside $s(\rho,\theta)$ imposes some difficulty.  Moreover, the system possesses diffusion in $\theta$ but not in $\rho$.  Nevertheless, when the functions involved have sufficient regularity, we can use the product and chain rules to obtain a suitable variant.  We also remark that the De Giorgi technique does not apply to general systems; indeed, counterexamples (due to De Giorgi) to the corresponding regularity results exist.

\vspace{0.2in}

Note that the regularity required to perform this procedure is only present at the very beginning of the approximation procedure described in \cite{FeireislNovotnyBook}, where the equation has a number of additional terms which would interfere with the De Giorgi argument.  Indeed, the existence of such smooth solutions for the original system $(\ref{nsf_system})$ is a major open question.  In light of this, we first establish the local entropy inequality for smooth solutions to the Navier-Stokes-Fourier system ($\ref{nsf_system}$).  Our main result, Theorem $\ref{thm_temp1}$, then imposes this inequality as an assumption used to derive the desired temperature bounds.  We refer to the section below on the existence theory for weak solutions for further comments on this issue.

\vspace{0.2in}

\begin{center}
{\bf Notion of weak solution}
\end{center}

As mentioned above, we consider a weak formulation of the system $(\ref{nsf_system})$, based upon the existence theory developed by Feireisl, Novotn\'y et al. for a related system with a additional radiative terms.  We now recall the relevant notion of weak solution from \cite[Section $2.1$]{FeireislNovotnyBook}, written for the system ($\ref{nsf_system}$).

Suppose that $\mathbb{S}$, $\kappa$, $\sigma$, $s$ and $p$ satisfy the consitutive relations established in Section $2$ below.  We say that a triple of measurable functions $(\rho,u,\theta)$ is then a {\it weak solution} of the Navier-Stokes-Fourier system $(\ref{nsf_system})$ if $\rho\in L^1((0,T)\times \Omega)$, $\ebdiv u\in L^1((0,T)\times \Omega)$, and, for some $q>1$,  $\nabla u\in L^1(0,T;L^q(\Omega;\mathbb{R}^{3\times 3}))$, $(\theta,\nabla\theta)\in L^q((0,T)\times \Omega)^2$, with 
\begin{enumerate}
\item[(i)] $\rho\geq 0$, $\theta\geq 0$ a.e. on $(0,T)\times \Omega$,
\item[(ii)] $u|_{\partial\Omega}=0$, and 
\item[(iii)] the continuity equation is satisfied in the renormalized sense of \cite{DiPerna-Lions}; that is,
\begin{align}
\nonumber &\int_0^T\int_{\Omega} \rho B(\rho)(\partial_t\phi^{(1)}+u\cdot \nabla \phi^{(1)})dxdt\\
\label{eq-weak-continuity} &\hspace{0.8in}=\int_0^T\int_{\Omega} b(\rho)\ebdiv u\phi^{(1)} dxdt-\int_{\Omega} \rho_0 B(\rho_0)\phi^{(1)}(0)dx,\\
\intertext{
for every $b\in L^\infty\cap C^0(0,\infty)$, $B:[0,\infty)\rightarrow\mathbb{R}$ defined by 
\begin{align*}
B(\rho)=B_1+\int_1^\rho \frac{b(z)}{z^2}dz,
\end{align*}
and all test functions $\phi^{(1)}\in C^1_c([0,T)\times \overline{\Omega})$, while the momentum and entropy production equations are satisfied in the distributional sense; that is,
}
\nonumber &\int_0^T\int_{\Omega} \rho u\cdot \partial_t\phi^{(2)}+\rho[u\otimes u]:\nabla \phi^{(2)}+p\ebdiv \phi^{(2)} dxdt\\
&\hspace{0.8in}=\int_0^T\int_{\Omega} \mathbb{S}:\nabla\phi^{(2)} dxdt-\int_{\Omega} (\rho u)_0\cdot \phi^{(2)}(0)dx\label{eq-weak-momentum}
\intertext{and}
\nonumber &\int_0^T\int_{\Omega} \rho s(\partial_t\phi^{(3)}+u\cdot \nabla \phi^{(3)})dxdt+\int_0^T\int_{\Omega} \frac{-\kappa\nabla\theta}{\theta}\cdot \nabla\phi^{(3)} dxdt\\
&\hspace{0.8in}=-\langle \sigma,\phi^{(3)}\rangle-\int_{\Omega} \rho_0s(0)\phi^{(3)}(0)dx,\label{eq-weak-entropy-prod}
\end{align}
for  all test functions $\phi^{(2)}\in C^1_c([0,T)\times \overline{\Omega};\mathbb{R}^3)$, $\phi^{(3)}\in C^1_c([0,T)\times \overline{\Omega})$, together with the integrability conditions required to make sense of each quantity in ($\ref{eq-weak-continuity}$), ($\ref{eq-weak-momentum}$) and ($\ref{eq-weak-entropy-prod}$).
\end{enumerate}

\vspace{0.2in}

\begin{center}
{\bf Comments on the existence theory for weak solutions}
\end{center}

As remarked above, the existence of weak solutions (in the sense described in the previous section) satisfying the hypotheses of Theorem $\ref{thm_temp1}$ is not known at present.  In this context, two distinct issues arise: first, the existence of weak solutions for (\ref{nsf_system}) itself in the specific setting of the class of constitutive relations described in Section $2$ and second, existence of weak solutions satisfying the additional hypotheses identified in the statement of Theorem $\ref{thm_temp1}$.

Concerning the first issue, existence of weak solutions for (\ref{nsf_system}), recent work of Feireisl and Novotn\'y \cite{FeireislBook,FeireislIUMJ53_2004,FeireislContempMath371,FeireislCommPDE31_2006,FeireislNovotnyProcRSE135A_2005,FeireislNovotnyBook} have developed an existence theory for weak solutions of (\ref{nsf_system}) when the constitutive relations on the heat conduction, entropy, internal energy and pressure adhere to the hypotheses described in Section $2$ and, moreover, admit an additional term describing the influence of radiation at high temperatures.  It should be noted that this radiative term is currently required for the existence theory.  However, at the present time, it is not clear how to adapt the proof of Theorem $\ref{thm_temp1}$ to allow for the presence of radiation, and we therefore consider the non-radiative case (for which existence of weak solutions is at present an open question).

Turning to the second issue, existence of weak solutions satisfying the full hypotheses of Theorem $\ref{thm_temp1}$, the additional assumptions beyond the notion of weak solution amount to integrability for $u$ and $\rho$, and the local entropy inequality ($\ref{eq_admiss}$).  Note that the integrability condition $u\in L^2(0,T;H_0^1(\Omega))$ can be ensured by taking $\alpha=1$ in ($\ref{eq_mu}$) and ($\ref{eq_kappa}$) (see \cite[Theorem $3.2$]{FeireislNovotnyBook}), while the bound $\rho\in L^\infty(0,T;L^\omega(\Omega))$ for some $\omega>3$ can be imposed by adding an additional term to the pressure; the desired bounds then follow for this adjusted equation via the energy inequality (see for instance the treatment in \cite{FeireislBook}).

Concerning the local entropy inequality, it is reasonable to expect that the arguments we present can be further developed, adapting the proof of existence to preserve the local entropy inequality in the limit (corresponding to existence of suitable weak solutions obtained by Caffarelli, Kohn and Nirenberg in \cite{CKN}); such an approach is carried out for a compressible system without heat conduction in \cite{FNsuitable}.  However, we choose not to pursue these issues further here.

\vspace{0.2in}

\begin{center}
{\bf Outline of the paper}
\end{center}

We now give a brief outline of the rest of the paper.  In Section $2$, we establish some notation, fix our assumptions on the constitutive relations, and give the formal statement of the main results of our study.  Sections $3$ and $4$ are then devoted to the proofs of the local entropy inequality and the bounds from below on the temperature, respectively.  We conclude with a brief appendix giving a basic distributional calculation that will be useful for our arguments, and describing how an additional hypothesis of bounded density can lead to some relaxation in the growth hypotheses imposed on the entropy.

\section{Constitutive relations and general assumptions on the system}
\label{prelim}
We now introduce some hypotheses that further restrict the constitutive assumptions for the system ($\ref{nsf_system}$).  In particular, in the remainder of the paper we will assume that $p,e,s\in C^1((0,\infty)\times (0,\infty))$, $\mu,\eta\in C^1([0,\infty))$ and $\kappa\in C^1([0,\infty))$, $P\in C^1([0,\infty))$ satisfy the hypotheses listed below.

We begin by stating some structural hypotheses concerning the influence of viscosity and heat-conduction within the fluid.  In particular, we will assume that $\mathbb{S}$ and $\sigma$ take the form
\begin{align}
\nonumber \mathbb{S}&=\mu(2D(u)-\tfrac{2}{3}I\ebdiv u)+\eta I\ebdiv u\\
\sigma&\geq \frac{\mu|2D(u)-\frac{2}{3}I\ebdiv u|^2+2\eta |\ebdiv u|^2}{2\theta}+\frac{\kappa |\nabla\theta|^2}{\theta^2}\label{sigma_def}
\end{align}
with $D(u)=\nabla u+(\nabla u)^\top$ and where $\sigma$ is a Borel measure on $[0,T]\times\overline{\Omega}$.

Because we are working in a compressible model, the forces driving the fluids evolution include the pressure that the fluid exerts upon itself, in addition to the viscous interactions described by the stress tensor $\mathbb{S}$ above.  The derivation of these forces arises from thermodynamical considerations, beginning with Gibbs' equation, 
\begin{align}
\label{gibbs}\theta D_{(\rho,\theta)}s(\rho,\theta)&=D_{(\rho,\theta)}e(\rho,\theta)+p(\rho,\theta)D_{(\rho,\theta)}(\frac{1}{\rho}),\qquad \textrm{for}\quad\rho,\theta>0,
\end{align}
where $D_{(\rho,\theta)}=(\partial_\rho,\partial_\theta)$.
As mentioned above, the quantities $s$ and $p$ represent the entropy and pressure of the system, while $e$ represents the internal energy.  %They are related by the conditon ($\ref{gibbs}$), which is known as Gibbs' equation and is based on the second law of thermodynamics.
Regarding $p$ and $e$, we require that for all $\rho>0$, there exists $\underline{e}(\rho)>0$ such that $\lim_{\theta\rightarrow 0^+} e(\rho,\theta)=\underline{e}(\rho)$, and that for all $\rho,\theta>0$ one has $\partial_\rho p(\rho,\theta)>0$, $0<\partial_\theta e(\rho,\theta)\leq c$ and $|\rho\partial_\rho e(\rho,\theta)|\leq ce(\rho,\theta)$.  Moreover, for the purposes of our study we restrict ourselves to the study of a monoatomic gas in the absence of thermal radiation effects, in which we have the further relation
\begin{align}
p(\rho,\theta)=\tfrac{2}{3}\rho e(\rho,\theta).\label{monatomic}
\end{align}

As a consequence of $(\ref{gibbs})$ and $(\ref{monatomic})$, there exists $P\in C^1$ such that $P(0)=0$, $P'(0)>0$, and 
\begin{align*}
p(\rho,\theta)=\theta^\frac{5}{2}P(\frac{\rho}{\theta^\frac{3}{2}})
\end{align*}
for all $\rho,\theta>0$.  In accordance with ($\ref{gibbs}$) and the above hypotheses on $p(\rho,\theta)$ and $e(\rho,\theta)$, we have
\begin{align}
s(\rho,\theta)=S(\frac{\rho}{\theta^{3/2}})\quad\textrm{with}\quad S'(Z)=-\frac{3}{2}\left(\frac{\frac{5}{3}P(Z)-ZP'(Z)}{Z^2}\right).\label{eqAA1}
\end{align}
Moreover, these hypotheses on $p(\rho,\theta)$ and $e(\rho,\theta)$ ensure that 
\begin{align}
S'(Z)\geq -c_1Z^{-1}, \quad Z>0.\label{eqAA2} 
\end{align}

Finally, concerning the shear viscosity and heat conduction coefficients $\mu$ and $\kappa$, we shall assume that for some $\alpha\in (\frac{2}{5},1]$, the conditions
\begin{align}
(1+\theta^\alpha)\underline{\mu}\leq \mu(\theta)\leq (1+\theta^\alpha)\overline{\mu},\label{eq_mu}
\end{align}
\begin{align*}
\sup_\theta |\mu'(\theta)|\leq \overline{m}
\end{align*}
and
\begin{align}
\underline{\kappa}\leq \kappa(\theta)\leq \overline{\kappa}(1+\theta^\alpha)\label{eq_kappa}
\end{align}
hold for some constants $0<\underline{\mu}<\overline{\mu}<\infty$, $0<\underline{\kappa}<\overline{\kappa}<\infty$ and $\overline{m}>0$.

\vspace{0.2in}

We remark that all of the above assumptions are physical, internally consistent, and also consistent with the work \cite{FeireislNovotnyBook} (see also \cite{FeireislBook}), up to the exclusion of the radiative term as discussed above.  For technical reasons, we will also impose two additional constraints:
\begin{enumerate}
\item[(i)] the inequality $\theta\leq \eta(\theta)$ holds for $\theta$ sufficiently small, and 
\item[(ii)] there exists $C_2>0$ such that
\begin{align}
S'(Z)\leq -C_2Z^{-1}\quad\forall Z>0.\label{growth1}
\end{align}
\end{enumerate}
The constraint (i) is a statement of non-degeneracy of the viscosity coefficient which enables us to use the diffusion in $\theta$ to control the growth of a quantity like $\log\theta$, while the constraint (ii) ensures that the entropy grows sufficiently fast as $\theta$ tends to zero.  Note that the case of affine pressure treated in \cite{MelletVasseur} corresponds to $S'(Z)=-Z^{-1}$.  We refer to Appendix B for a discussion of how $(\ref{growth1})$ can be relaxed when the density is known to remain bounded.

Having established these assumptions on the constitutive relations, we now address the main results of our study.
\section{The local entropy inequality for smooth solutions}

We now turn to the local entropy inequality ($\ref{eq_admiss}$) that we described in the introduction.  In particular, the statement of this inequality will make strong use of the following truncation operator, which will be applied to the temperature $\theta$: for $C>0$, we define $f_{C}:[0,\infty)\rightarrow [0,C)$ by
\begin{align*}
f_C(z)&=(z-C)_-+C=\min \{z,C\}.
\end{align*}
Fixing $0<\mathfrak{s}<t<T$, the local entropy inequality is then
\begin{align}
\nonumber &\int_{\Omega} \rho\tilde{s}_C(t,x)dx+\int_{\mathfrak{s}}^t\int_{\{\theta\leq C\}}\frac{\mu|2D(u)-\frac{2}{3}I\ebdiv u|^2+2\eta|\ebdiv u|^2}{2\theta}+\frac{\kappa|\nabla\theta|^2}{\theta^2}dxdt\\
&\hspace{0.2in}\leq \int_s^t\int_{\{\theta\leq C\}} \left(-\rho^2\partial_{\rho} s(\rho,C)\ebdiv u\right)dxdt+\int_{\Omega} \rho\tilde{s}_C(\mathfrak{s},x)dx,\label{eq_admiss}
\end{align}
where
\begin{align*}
\tilde{s}_C=s(\rho,C)-s(\rho,f_C(\theta)).
\end{align*}
In particular, smooth solutions to the Navier-Stokes-Fourier system ($\ref{nsf_system}$) satisfy $(\ref{eq_admiss})$:
\begin{proposition}
\label{thm_admissible}
Fix $m\in\mathbb{N}$.  If $(\rho,u,\theta)$ is a smooth solution to the Navier-Stokes-Fourier system ($\ref{nsf_system}$) with $\rho\in L^\infty(0,T;L^\omega(\Omega))$ for some $\omega>3$ and $u\in L^2(0,T;H^1_0(\Omega))$, then for every $0\leq \mathfrak{s}\leq t<\infty$ and $C>0$ the solution satisfies the local entropy inequality $(\ref{eq_admiss})$.
\end{proposition}
It should be noted that the existence of such smooth solutions is an oustanding open question.  Nevertheless, Proposition $\ref{thm_admissible}$ indicates the plausability of imposing $(\ref{nsf_system})$ as an additional restriction on the class of weak solutions.

\begin{proof}[Proof of Proposition $\ref{thm_admissible}$]
Let $C>0$ be given and note that
\begin{align}
\partial_t (\rho \tilde{s}_C)+\ebdiv(\rho \tilde{s}_Cu)&=(I)-(II),\label{ref1}
\end{align}
where we have set
\begin{align*}
(I)&:=\partial_t (\rho s(\rho,C))+\ebdiv (\rho s(\rho,C)u)
\end{align*}
and
\begin{align*}
(II)&:=\partial_t (\rho s(\rho, f_C(\theta)))+\ebdiv(\rho s(\rho, f_C(\theta))u),\\
\end{align*}
Note that a straightforward calculation gives
\begin{align*}
(I)&=(\partial_t \rho)s(\rho,C)+\rho\partial_t s(\rho,C)+\nabla s(\rho,C)\cdot (\rho u)+s(\rho,C)\ebdiv (\rho u)\\
&=\rho (\partial_\rho s(\rho,C)\partial_t\rho)+(\partial_\rho s(\rho,C)\nabla\rho)\cdot (\rho u)\\
&=-\rho^2\partial_\rho s(\rho,C)\ebdiv u,
\end{align*}
where we have used the continuity equation $\partial_t \rho+\ebdiv(\rho u)=0$ to obtain both the second and third equalities.

For $(II)$, we will make use of the identity
\begin{align}
\partial_t\theta+\nabla\theta\cdot u&=\frac{1}{\alpha(\rho,\theta)}\bigg[\partial_t (\rho s)+\ebdiv(\rho s u)+\rho^2 \partial_\rho s(\rho,\theta)\ebdiv u\bigg],\label{eqident}
\end{align}
where we have set $\alpha(\rho,\theta)=\rho \partial_\theta s(\rho,\theta)$.  Indeed, using the definition of $s$ and the product rule, we obtain
\begin{align*}
&\partial_t(\rho s)+\ebdiv (\rho su)\\
&\hspace{0.2in}=\partial_t (\rho s(\rho,\theta))+\ebdiv (\rho s(\rho,\theta)u)\\
&\hspace{0.2in}=(\partial_t \rho)s(\rho,\theta)+\rho \partial_t s(\rho,\theta)+\nabla s(\rho,\theta)\cdot (\rho u)+s(\rho,\theta)\ebdiv (\rho u).
\end{align*}
This is then equal to
\begin{align*}
&\rho \bigg[\partial_\rho s(\rho,\theta)\partial_t\rho+\partial_\theta s(\rho,\theta)\partial_t \theta\bigg]+\bigg[ \partial_\rho s(\rho,\theta)\nabla \rho+\partial_\theta s(\rho,\theta)\nabla\theta\bigg]\cdot (\rho u)\\
&\hspace{0.4in}=\rho\partial_\rho s(\rho,\theta)(\partial_t\rho+\nabla\rho \cdot u)+(\rho\partial_\theta s(\rho,\theta))(\partial_t\theta+\nabla\theta\cdot u),
\end{align*}
which gives the identity.

Returning to $(II)$, we use the product rule to obtain,
\begin{align*}
(II)&=(\partial_t \rho)s(\rho,f_C(\theta))+\rho \partial_t s(\rho,f_C(  \theta ))+\nabla s(\rho,f_C(  \theta ))\cdot (\rho u)\\
&\hspace{0.2in}+s(\rho,f_C(  \theta ))\ebdiv (\rho u)\\
&=\rho\partial_\rho s(\rho,f_C(\theta))(\partial_t\rho+u\cdot \nabla\rho)+\rho f'_C(\theta)\partial_\theta s(\rho,f_C(\theta))(\partial_t\theta+u\cdot \nabla\theta)\\
&=-\rho^2\partial_\rho s(\rho,f_C(\theta))\ebdiv u+\rho f'_C(\theta)\partial_\theta s(\rho,f_C(\theta))(\partial_t\theta+u\cdot \nabla\theta)
\end{align*}
so that
\begin{align*}
(II)&=-\rho^2\partial_\rho s(\rho,\theta)\ebdiv u+\rho\partial_\theta s(\rho,\theta)(\partial_t \theta+u\cdot\nabla\theta)
\end{align*}
when $\theta\leq C$ and $(II)=(I)$ when $\theta>C$.

Combining these calculations, we obtain that $(\ref{ref1})$ is equal to
\begin{align*}
&\rho^2(\partial_\rho s(\rho,\theta)-\partial_\rho s(\rho,C))\ebdiv u\\
&\hspace{0.2in}-(\rho \partial_\theta s(\rho, \theta))(\partial_t\theta+u\cdot \nabla \theta)\\
&\hspace{0.4in}=(-\rho^2\partial_\rho s(\rho,C))\ebdiv u-\left[\partial_t(\rho s)+\ebdiv(\rho s u)\right]
\end{align*}
for $(t,x)$ such that $\theta\leq C$, and equal to $0$ when $\theta>C$.

We then have
\begin{align*}
&\partial_t (\rho\tilde{s}_C)+\ebdiv (\rho \tilde{s}_Cu)\\
&\hspace{0.2in}\leq \bigg[ -\rho^2\partial_\rho s(\rho ,C)\ebdiv u\\
&\hspace{0.5in}-\frac{\mu|2D(u)-\frac{2}{3}I\ebdiv u|^2+2\eta |\ebdiv u|^2}{2\theta}-\frac{\kappa|\nabla\theta|^2}{\theta^2}\bigg]1_{\{\theta\leq C\}}(\theta)\\
&\hspace{0.5in}-\ebdiv\left(\frac{\kappa\nabla\theta}{\theta}\right)1_{\{\theta\leq C\}}(\theta)\\
&\hspace{0.2in}\leq \bigg[ -\rho^2\partial_\rho s(\rho ,C)\ebdiv u\\
&\hspace{0.5in}-\frac{\mu|2D(u)-\frac{2}{3}I\ebdiv u|^2+2\eta |\ebdiv u|^2}{2\theta}-\frac{\kappa|\nabla\theta|^2}{\theta^2}\bigg]1_{\{\theta\leq C\}}(\theta)\\
&\hspace{0.5in}-\ebdiv\left(\frac{\kappa \nabla f_C(\theta)}{\theta}\right)
\end{align*}
in the sense of distributions, where we have used ($\ref{nsf_system}$), ($\ref{sigma_def}$) and Lemma $\ref{lem_app}$.\footnote{In fact, solutions of ($\ref{nsf_system}$) satisfy ($\ref{sigma_def}$) with equality.  We retain the inequality in our calculation to emphasize that ($\ref{eq_admiss}$) is consistent (at least formally) with ($\ref{sigma_def}$) under weaker notions of solution, provided that other steps in the argument can be given proper justification.}  The desired result follows by integrating over $[\mathfrak{s},t]\times\Omega$.
\end{proof}

\section{Temperature bounds: the proof of Theorem $\ref{thm_temp1}$}

We next turn to the proof of Theorem $\ref{thm_temp1}$.  Recall that the goal of this theorem is to establish uniform bounds from below on the temperature $\theta$ for weak solutions satisfying the local entropy inequality $(\ref{eq_admiss})$ (together with certain integrability conditions on $\rho$ and $u$).

The proof of this result follows the proof of \cite[Theorem $1$]{MelletVasseur} and, as we mentioned above, is based on the use of Stampacchia trunactions and De Giorgi's regularity theory for elliptic partial differential equations.  We remark that these methods have seen much recent application in parabolic problems and the equations of fluid mechanics; see for instance \cite{CaffarelliVasseur} and the references cited there - we also point out the works of Caffarelli et al. \cite{CVAnnals}, Beir\"ao da Veiga \cite{BeiraoDaVeiga} and Chan \cite{Chan}, as well as a treatment of the partial regularity theory \cite{VasseurNoDEA}.

To facilitate the De Giorgi iteration argument, we recall a lemma showing how superlinear bounds can lead to improved convergence properties.
\begin{lemma}
\label{lem_ineq}
Let $C>1$ and $\beta>1$ be given and let $(W_k)_{k\in\mathbb{N}}$ be a sequence in $[0,1]$ such that for every $k\in\mathbb{N}$, $W_{k+1}\leq C^{k+1}W_k^\beta$.  Then there exists $C_0^*$ such that $0<W_1<C_0^*$ implies $W_k\rightarrow 0$ as $k\rightarrow\infty$.
\end{lemma} 
The estimate contained in Lemma $\ref{lem_ineq}$ is classical; for a proof, see for instance \cite{VasseurNoDEA}.  With this lemma in hand, we now address the proof of the theorem:
\begin{proof}[Proof of Theorem $\ref{thm_temp1}$]
Let $\Omega\subset\mathbb{R}^3$, $T>0$, and $(\rho,u,\theta)$ be given as stated.  Fix a decreasing sequence $(C_k)_{k\geq 0}\subset\mathbb{R}^+$ and an increasing sequence $(T_k)_{k\geq 0}\subset \mathbb{R}^+$, both to be chosen later in the argument, and define
\begin{align}
\nonumber U_k&:=U(C_k,T_k)
\end{align}
where for each $C>0$, $s>0$, we have set
\begin{align}
\nonumber U(C,\mathfrak{s})&:=\esssup_{\mathfrak{s}\leq t\leq T} \int_\Omega \rho(t,x)\log\left(C/f_{C}(\theta(t,x))\right)dx\\
\nonumber &\hspace{0.2in}+\int_{\mathfrak{s}}^T \int_{\Omega} \frac{\eta}{\theta}|\ebdiv u|^2\chi_{\theta\leq C}(t,x)dxdt\\
&\hspace{0.2in}+\int_{\mathfrak{s}}^T \int_{\Omega}\frac{\kappa |\nabla \theta|^2}{\theta^2}\chi_{\theta\leq C}(t,x)dxdt.\label{defu}
\end{align}
and where $\chi_{\theta\leq C}=\chi_{\{(t,x)\in [0,T]\times\Omega:\theta(t,x)\leq C\}}$.

\vspace{0.2in}

\noindent
{\bf Step $1$}: {\it Boundedness of $U_{k+1}$.}

\vspace{0.2in}

\noindent 
Note that $\tilde{s}_{C_{k+1}}=0$ on $\{(t,x):\theta \geq C_{k+1}\}$, while on the set $\{\theta<C_{k+1}\}$, we use ($\ref{growth1}$) to estimate $\tilde{s}_{C_{k+1}}$, obtaining
\begin{align*}
s(\rho,C_{k+1})-s(\rho,\theta)&=\int_{\theta}^{C_{k+1}} (\partial_\theta s)(\rho,\omega)d\omega=\int_{\theta}^{C_{k+1}} -\frac{3\rho}{2\omega^{5/2}}S'(\frac{\rho}{\omega^{3/2}})d\omega\\
&\geq c\int_{\theta}^{C_{k+1}} \frac{3}{2\omega}d\omega
=c\log\left(C_{k+1}/\theta\right).
\end{align*}

Invoking the local entropy inequality $(\ref{eq_admiss})$ and recalling $\rho_0(x)=\rho(0,x)$, $\theta_0(x)=\theta(0,x)$, we therefore get the inequality
\begin{align*}
U_{k+1}&\leq C\left(\int_{0}^T\int_{\Omega}\chi_{\theta\leq C_{k+1}}\rho^2(-\partial_\rho s(\rho,C_{k+1}))|\ebdiv u|dxdt\right.\\
&\hspace{2.6in}\left.+\int_{\Omega} \rho_0\tilde{s}_{C_{k+1}}(\rho_0,\theta_0)dx\right)
\end{align*}
Now, making use of ($\ref{eqAA1}$) and ($\ref{eqAA2}$), we obtain
\begin{align}
\nonumber U_{k+1}&=C\int_0^T\int_{\Omega} \chi_{\theta\leq C_{k+1}}\frac{\rho^2}{C_{k+1}^{3/2}}S'(\frac{\rho}{C_{k+1}^{3/2}})|\ebdiv u|dxdt\\
\nonumber &\hspace{0.2in}-\int_{\Omega}\int_{f_{C_{k+1}}(\theta_0)}^{C_{k+1}} \frac{3\rho_0^2}{2\omega^{5/2}}S'(\frac{\rho_0}{\omega^{3/2}})d\omega dx\\
\nonumber &\leq C\int_0^T\int_{\Omega} \chi_{\theta\leq C_{k+1}}\rho(t,x)|\ebdiv u|dxdt\\
&\hspace{0.2in}+\frac{3}{2}\int_{\Omega}\rho_0\log(C_{k+1}/f_{C_{k+1}}(\theta_0))dx\label{eqaaa1}
\end{align}
Using H\"older in the first term followed by the hypotheses $\rho\in L^\infty(0,T;L^\omega(\Omega))$, $u\in L^2(0,T;H_0^1(\Omega))$ and ($\ref{eqAA3}$), we therefore obtain
\begin{align}
U_{k+1}&\leq C^*\label{cstar}
\end{align}
for some $C^*>0$.

\vspace{0.2in}

\noindent {\bf Step $2$}: {\it Local entropy estimate for $U_{k+1}$.} 

\vspace{0.2in}

Arguing as above, we again invoke the local entropy inequality $(\ref{eq_admiss})$ and expand the interval of integration to $[T_{k},T_{k+1}]$ (using $-\partial_\rho s\geq 0$), which gives the estimate
\begin{align*}
U_{k+1}&\leq C\left(\int_{T_k}^{T}\int_{\Omega}\chi_{\theta\leq C_{k+1}}\rho^2(-\partial_\rho s(\rho,C_{k+1}))|\ebdiv u|dxdt\right.\\
&\hspace{2.6in}\left.+\int_{\Omega} \rho\tilde{s}_{C_{k+1}}(\mathfrak{s},x)dx\right)
\end{align*}
for a.e. $T_k\leq s\leq T_{k+1}$.  Integrating both sides of this inequality over $\mathfrak{s}\in [T_k,T_{k+1}]$ and dividing by $T_{k+1}-T_k$, we obtain
\begin{align*}
U_{k+1}&\leq C\left(\int_{T_k}^T\int_{\Omega} \chi_{\theta\leq C_{k+1}}\rho^2(-\partial_\rho s(\rho,C_k)|\ebdiv u|)dxdt\right.\\
&\hspace{1.6in}\left.+\frac{1}{T_{k+1}-T_k}\int_{T_{k}}^{T_{k+1}}\int_{\Omega} \rho\tilde{\mathfrak{s}}_{C_{k+1}}(0,x)dxd\mathfrak{s}\right)
\end{align*}

Arguing as in $(\ref{eqaaa1})$ and recalling that the sequence $(C_k)$ is decreasing, the Cauchy-Schwarz inequality gives the bound 
\begin{align*}
&\int_{T_k}^T\int_{\Omega} \chi_{\theta\leq C_{k+1}}\rho^2(-\partial_\rho s(\rho,C_{k+1}))|\ebdiv u|dxdt\\
&\hspace{0.2in}\leq C\left\lVert \frac{\eta^{1/2}\ebdiv u}{\theta^{1/2}}\chi_{\theta\leq  C_{k}}\right\rVert_{L^2([T_k,T];L^2(\Omega))}\left\lVert \rho\chi_{\theta\leq C_{k+1}}\right\rVert_{L^2([T_k,T];L^2(\Omega))}\\
&\hspace{0.2in}\leq CU_k^{1/2}\left(\int_{T_{k}}^T\int_{\Omega} \rho(t,x)^2\chi_{\theta \leq C_{k+1}}(t,x)dxdt\right)^{1/2},
\end{align*}
where we have used the constraint (i) appearing at the end of Section $2$.  This in turn gives
\begin{align}
U_{k+1}&\leq C\left[U_k^{1/2}(I)^{1/2}+(II)\right]\label{eq0bd}
\end{align}
with
\begin{align*}
(I)&:=\int_{T_{k}}^T\int_{\Omega} \rho(t,x)^2\chi_{\theta \leq C_{k+1}}(t,x)dxdt,\\
(II)&:=\frac{1}{T_{k+1}-T_k}\int_{T_k}^{T_{k+1}}\int_{\Omega} \rho\tilde{s}_{C_{k+1}}(\mathfrak{s},x)dxd\mathfrak{s}.
\end{align*}
The next two steps of the arugment consist of estimating the terms $(I)$ and $(II)$.

\vspace{0.2in}

\noindent {\bf Step $3$}: {\it Tchebyshev estimates for (I).}

\vspace{0.2in}

\noindent Define
\begin{align*}
F_k(\theta)&:=\chi_{\theta\leq C_k}\log\left(C_k/\theta\right),\\
R_k&:=\log\left(C_k/C_{k+1}\right),
\end{align*}
and observe that $(C_k)$ decreasing implies that the inequality $R_k\leq F_k(\theta(t,x))$ holds on the set $\{\theta<C_{k+1}\}$.  Fix parameters $\alpha$, $\beta$, $p$ and $q$ satisfying  
\begin{align}
\alpha\in (0,2),\quad \beta>0\quad\textrm{and}\quad p,q\geq 1\label{req1}
\end{align}
to be determined later in the argument, and let $p'$ and $q'$ be the conjugate exponents to $p$ and $q$.  

Using H\"older, we obtain the estimate
\begin{align}
\nonumber (I)&\leq \frac{1}{R_k^\beta}\int_{T_k}^T\int_{\Omega} \rho(t,x)^2F_k(\theta(t,x))^\beta dxdt\\
\nonumber &\leq \frac{1}{R_k^\beta}\lVert \rho\rVert^{2-\alpha}_{L^{(2-\alpha)p}([T_k,T];L^{(2-\alpha)q}(\Omega))}\lVert \rho^{\alpha} F_k(\theta)^\beta\rVert_{L^{p'}([T_k,T];L^{q'}(\Omega))}\\
&\leq \frac{C(T,|\Omega|,\alpha,p,q)}{R_k^\beta}\lVert \rho\rVert_{L^\infty(L^{\omega}(\Omega))}^{2-\alpha}\lVert \rho^{\alpha} F_k(\theta(t,x))^\beta\rVert_{L^{p'}([T_k,T];L^{q'}(\Omega))}\label{eqaa1}
\end{align}
provided that $\alpha$ and $q$ satisfy 
\begin{align}
(2-\alpha)q<\omega.\label{req2}
\end{align}

We now turn to the task of estimating $\lVert \rho^\alpha F_k^\beta\rVert_{L^{p'}L^{q'}}$.  In particular, we obtain
\begin{align}
\nonumber \lVert \rho^{\alpha} F_k^\beta\rVert_{L^{p'}L^{q'}}&=\lVert (\rho F_k)^{\alpha/\beta}F_k^{1-\frac{\alpha}{\beta}}\rVert_{L^{\beta p'}L^{\beta q'}}^\beta\\
\nonumber &\leq \lVert (\rho F_k)^{\alpha/\beta}\rVert_{L^{\infty}L_x^{\beta/\alpha}}^\beta\lVert F_k^{1-\frac{\alpha}{\beta}}\rVert_{L^\frac{2}{1-\frac{\alpha}{\beta}}L^\frac{6}{1-\frac{\alpha}{\beta}}}^\beta\\
\nonumber &=\lVert \rho F_k\rVert_{L^\infty L^1}^\alpha\lVert F_k\rVert_{L^2L^6}^{\beta-\alpha}\\
&\leq CU_{k}^\alpha\lVert F_k\rVert_{L^2L^6}^{\beta-\alpha}\label{eqaa2}
\end{align}
where we have set $p'=\frac{2}{\beta-\alpha}$ and $q'=\frac{6}{5\alpha+\beta}$, i.e.
\begin{align}
p=\frac{2}{2+\alpha-\beta}, \quad q=\frac{6}{6-5\alpha-\beta}.\label{req3}
\end{align}

The estimate of $\lVert F_k\rVert_{L^2L^6}$ is based the following inequality of Sobolev type adapted to the norms appearing in $U_k$, which we recall from \cite{MelletVasseur}.
\begin{lemma}[Sobolev-type inequality, \cite{MelletVasseur}]
\label{lem4}
Let $\Omega\subset\mathbb{R}^3$ be a bounded domain with smooth boundary.  Given $T>0$ and $\rho\in L^\infty([0,T];L^\omega(\Omega))$ for some $\omega>3$ such that
\begin{align*}
t\mapsto \int_{\Omega} \rho(t,x)dx
\end{align*}
is constant in $t$, there exists $C=C(\Omega,T,\rho,\omega)>0$ such that the inequality 
\begin{align*}
\lVert F\rVert_{L^2([0,T];L^6(\Omega))}\leq C(\lVert \rho F\rVert_{L^\infty([0,T];L^1(\Omega))}+\lVert \nabla F\rVert_{L^2([0,T];L^2(\Omega))})
\end{align*}
holds for every measurable $F:[0,T]\times\Omega\rightarrow [0,\infty)$,
\end{lemma}

Note that $\int \rho(t,x)dx=\int \rho_0(x)dx$ for a.e. $t\in [0,T]$.  Recalling that $\rho\in L^\infty([0,T];L^\omega(\Omega))$ is satisfied by hypothesis, we may therefore invoke Lemma $\ref{lem4}$ in our setting to obtain
\begin{align}
\nonumber \lVert F_k\rVert_{L^2L^6}^{\beta-\alpha}&\leq C\left(U_k+\lVert \chi_{\theta\leq C_k}\frac{\nabla\theta}{\theta}\rVert_{L^2L^2}\right)^{\beta-\alpha}\\
&\leq C\left(U_k+U_k^{1/2}\right)^{\beta-\alpha}.\label{eqaa3}
\end{align} 
Combining ($\ref{eqaa1}$) with ($\ref{eqaa2}$) and ($\ref{eqaa3}$) then gives
\begin{align}
(I)\leq \frac{C}{R_k^\beta}\left(U_k^\beta+U_k^{(\alpha+\beta)/2}\right).\label{eq1bd}
\end{align}

\vspace{0.2in}

\noindent {\bf Step $4$}: {\it Estimate for (II).}

\vspace{0.2in}

\noindent Arguing as in ($\ref{eqaaa1}$) and recalling that the sequence $(C_k)$ is decreasing, we note that ($\ref{eqAA1}$) and $(\ref{eqAA2})$ imply
\begin{align}
\tilde{s}_{C_{k+1}}(\rho(\mathfrak{s},x),\theta(\mathfrak{s},x))&\leq c\chi_{\theta\leq C_{k+1}}(\mathfrak{s},x)F_k(\theta(\mathfrak{s},x)).\label{eqAAAB1}
\end{align}
for a.e. $\mathfrak{s}\in [T_k,T_{k+1}]$ and a.e. $x\in \Omega$.

Invoking H\"older and arguing as in Steps $1$ and $2$ above, we obtain
\begin{align*}
(II)&\leq \frac{1}{T_{k+1}-T_k}\lVert F_{k}\rVert_{L^2([T_k,T_{k+1}];L^6(\Omega))}\lVert \rho\chi_{\theta\leq C_{k+1}}\rVert_{L^2([T_k,T_{k+1}];L^{6/5}(\Omega))}\\
&\leq \frac{1}{T_{k+1}-T_k}(U_k+U_k^{1/2})\lVert \rho\chi_{\theta<C_{k+1}}\rVert_{L^2([T_k,T_{k+1}];L^{6/5}(\Omega))}.
\end{align*}
On the other hand, proceeding as in Step $3$, we fix $\beta_1\in (0,1)$ to be determined later in the argument, and recall that $R_k\leq F_k(\theta)$ on $\{\theta<C_{k+1}\}$.  This yields
\begin{align*}
&\lVert \rho\chi_{\theta<C_{k+1}}\rVert_{L^2([T_k,T_{k+1}];L^{6/5}(\Omega))}\\
&\hspace{0.2in}=\lVert \rho^{6/5}\chi_{\theta<C_{k+1}}\rVert_{L^{5/3}([T_k,T_{k+1}];L^1(\Omega))}^{5/6}\\
&\hspace{0.2in}\leq \frac{C(T)}{R_k^{\frac{5\beta_1}{6}}}\lVert \rho^{6/5} F_k(\theta)^{\beta_1}\rVert_{L^{\infty}([T_k,T];L^{1}(\Omega))}^{5/6}\\
&\hspace{0.2in}\leq \frac{C(T)}{R_k^{\frac{5\beta_1}{6}}}\lVert \rho^{\frac{6}{5}-\beta_1}\rVert_{L^{\infty}([T_k,T];L^{1/(1-\beta_1)}(\Omega))}^{5/6}\lVert (\rho F_k(\theta))^{\beta_1}\rVert_{L^{\infty}([T_k,T];L^{1/\beta_1})}^{5/6}\\
&\hspace{0.2in}\leq \frac{C(T,|\Omega|,\alpha_1,q_1)}{R_k^{\frac{5\beta_1}{6}}}\lVert \rho\rVert_{L^\infty(L^\omega(\Omega))}^{1-5\beta_1/6}\lVert \rho F_k(\theta)\rVert_{L^{\infty}([T_k,T];L^{1})}^{5\beta_1/6}\\
&\hspace{0.2in}\leq \frac{C(T,|\Omega|,\alpha_1,q_1)}{R_k^{5\beta_1/6}}\lVert \rho\rVert_{L^\infty (L^\omega(\Omega))}^{1-5\beta_1/6}U_k^{5\beta_1/6},
\end{align*}
provided that $(\frac{6}{5}-\beta_1)/(1-\beta_1)<\omega$.  This in turn gives
\begin{align}
(II)&\leq \frac{C}{(T_{k+1}-T_k)R_k^{5\beta_1/6}}(U_k+U_k^{1/2})U_k^{5\beta_1/6}\label{eq2bd}
\end{align}

\vspace{0.2in}

\noindent{\bf Step $5$}: {\it Conclusion of the argument.}

\vspace{0.2in}

\noindent Combining ($\ref{eq0bd}$) with ($\ref{eq1bd}$) and ($\ref{eq2bd}$), we obtain
\begin{align}
U_{k+1}\leq \frac{C}{R_k^{\beta/2}}(U_k^{\frac{1+\beta}{2}}+U_k^{\frac{2+\beta+\alpha}{4}})+\frac{C}{(T_{k+1}-T_k)R_k^{5\beta_1/6}}(U_k^{1+\frac{5\beta_1}{6}}+U_k^{\frac{1}{2}+\frac{5\beta_1}{6}}) \label{finalUk}
\end{align}
whenever $\alpha,\beta,p$ and $q$ satisfy ($\ref{req1}$), ($\ref{req2}$), ($\ref{req3}$) and $\beta_1$, $q_1$ satisfy $\beta_1>0$, $q_1\geq 1$, $(\frac{6}{5}-\beta_1)/(1-\beta_1)<\omega$.

To complete the proof, we will use ($\ref{finalUk}$) (with an appropriate choice of parameters) and Lemma $\ref{lem_ineq}$ to conclude that $U_k\rightarrow 0$ as $k\rightarrow\infty$ for suitably chosen sequences $(C_k)$ and $(T_k)$, with limits $C_k\rightarrow C_\infty>0$ and $T_k\rightarrow \tau>0$, respectively.

Temporarily postponing the choice of $(C_k)$ and $(T_k)$, we remark that in order to apply Lemma $\ref{lem_ineq}$ the powers of $U_k$ appearing on the right side of ($\ref{finalUk}$) must be greater than $1$.  This is the primary motivation behind our choice of parameters; combining this requirement with ($\ref{req1}$) and ($\ref{req3}$), it suffices to choose $\alpha$ and $\beta$ satisfying 
\begin{align*}
|\beta-2|<\alpha<\frac{6-\beta}{5}, \quad \beta>1.
\end{align*}
This condition is compatible with ($\ref{req2}$) for $\omega>3$; we therefore choose such a pair $(\alpha,\beta)$.  

To choose $\beta_1$, we note that the condition $\frac{1}{2}+\frac{5\beta_1}{6}>1$ is satisfied for any $\beta_1>\frac{3}{5}$.  Moreover, the condition $(\frac{6}{5}-\beta_1)/(1-\beta_1)<\omega$ is satisfied for $\beta_1$ sufficiently close to $\frac{3}{5}$; choosing such a value of $\beta_1$, we see that ($\ref{finalUk}$) holds.

We now turn to the choice of the sequences $(C_k)$ and $(T_k)$.  Fix $M>1$ to be determined later in the argument and let $\tau\in [0,T]$ be given.  Now, setting
\begin{align*}
C_k=\exp(-M(1-2^{-k}))
\end{align*}
and
\begin{align*}
T_k=\tau(1-2^{-k}),
\end{align*}
and using Step $1$, we obtain
\begin{align*}
\frac{U_{k+1}}{C^*}&\leq \frac{C2^{(k+1)\beta/2}}{C^*M^{\beta/2}}U_k^{\gamma_1}+\frac{C2^{(k+1)(1+5\beta_1/6)}}{\tau C^*M^{5\beta_1/6}}U_k^{\gamma_2}\\
&=\frac{C2^{(k+1)\beta/2}(C^*)^{\gamma_1-1}}{M^{\beta/2}}\left(\frac{U_k}{C^*}\right)^{\gamma_1}\\
&\hspace{0.2in}+\frac{C2^{(k+1)(1+5\beta_1/6)}(C^*)^{\gamma_2-1}}{\tau M^{5\beta_1/6}}\left(\frac{U_k}{C^*}\right)^{\gamma_2}\\
&\leq C^{k+1}\left(\frac{U_k}{C^*}\right)^{\max\{\gamma_1,\gamma_2\}}
\end{align*}
for some $\gamma_1,\gamma_2>1$, where $C^*$ is as in ($\ref{cstar}$).  Invoking Lemma $\ref{lem_ineq}$, we find $C_0^*>0$ such that $U_{1}\leq C_0^*$ implies $U_k\rightarrow 0$ as $k\rightarrow\infty$.  On the other hand, by Step $1$, we have
\begin{align*}
U_1\leq \frac{C}{M^{\beta/2}}+\frac{C}{\tau M^{5\beta_1/6}}.
\end{align*}
Choosing $M$ sufficiently large, we obtain $U_k\rightarrow 0$ as desired.  We now conclude the proof as in \cite{MelletVasseur}.  In particular, taking the limit (since all integrands involved are nonnegative) we have
\begin{align*}
\int_\tau^T\int_{\Omega}\frac{\kappa |\nabla\theta|^2}{\theta^2}\chi_{\theta\leq e^{-M}}(t,x)dxdt\leq \liminf_{k\rightarrow\infty} \int_\tau^T\int_{\Omega} \frac{\kappa |\nabla\theta|^2}{\theta^2}\chi_{\theta\leq C_k}dxdt=0
\end{align*}
so that after observing the identity $\frac{|\nabla\theta|^2}{\theta^2}\chi_{\theta\leq e^{-M}}=|\nabla\log(e^{-M}/f_{e^{-M}}(\theta))|^2$, we obtain that $x\mapsto \log(e^{-M}/f_{e^{-M}}(\theta))$ is constant in $x$ for a.e. $t\in [\tau,T]$; that is, for a.e. $t$ we can find $A(t)\geq 0$ such that $A(t)=\log(e^{-M}/f_{e^{-M}}(\theta(t,x)))$ for a.e. $x\in \Omega$.  Taking the limit once more and using the weak (renormalized) form of the continuity equation $\partial_t\rho+\ebdiv(\rho u)=0$, we therefore have
\begin{align*}
0&=\int_\Omega \rho(t,x)\log(e^{-M}/f_{e^{-M}}(\theta(t,x)))dx\\
&=A(t)\int_{\Omega} \rho(t,x)dx\\
&=A(t)\int_{\Omega} \rho_0dx,\quad \textrm{a.e.}\quad t\in [\tau,T].
\end{align*}
We therefore obtain $A(t)\equiv 0$ for a.e. $t\in [\tau,T]$, which establishes
\begin{align*}
\theta(t,x)\geq e^{-M}
\end{align*}
for a.e. $t\in [\tau,T]$ and a.e. $x\in \Omega$.  This completes the proof of Theorem $\ref{thm_temp1}$.
\end{proof}

\appendix

\section{A distributional calculation}

In this brief appendix, we prove the following lemma, which is used in the proof of Proposition $\ref{thm_admissible}$.

\begin{lemma}
\label{lem_app}
Suppose that $\theta$ is smooth.  Then for every $C>0$, the inequality
\begin{align*}
-\ebdiv\left(\frac{\kappa\nabla\theta}{\theta}\right)\chi_{\theta\leq C}\leq -\ebdiv \left(\frac{\kappa\nabla\theta}{\theta}\chi_{\theta\leq C}\right)
\end{align*}
holds in the sense of distributions.
\end{lemma}
\begin{proof}
Let $\phi\in\mathcal{D}$ be given such that $\phi\geq 0$, and let $C>0$ be given.  Then
\begin{align*}
\langle -\ebdiv\left(\frac{\kappa\nabla\theta}{\theta}\right)\chi_{\{\theta\leq C\}},\phi\rangle&=-\int_{\{\theta\leq C\}}\ebdiv\left(\frac{\kappa\nabla\theta}{\theta}\right)\phi dx\\
&=\int_{\{\theta\leq C\}} \frac{\kappa\nabla\theta}{\theta}\cdot\nabla\phi dx-\int_{\partial\{\theta\leq C\}} \frac{\kappa\nabla\theta}{\theta}\phi\cdot \nu dS
\end{align*}
where $\nu$ is the unit outer normal to $\{\theta\leq C\}$ and $dS$ is the appropriate surface measure.  The smoothness of $\theta$ gives $\theta=C$ and $\nabla\theta\cdot\nu\geq 0$ on $\partial\{\theta\leq C\}$, so that we have
\begin{align*}
\int_{\partial\{\theta\leq C\}} \frac{\kappa\nabla\theta}{\theta}\phi\cdot \nu dS\geq 0,
\end{align*}
which gives the result.
\end{proof}

\section{Allowing slightly nonlinear growth: the case of bounded density.}

In this appendix, we show how an additional assumption of bounded density can enable us to allow slightly nonlinear growth in the function $S$ compared to the condition ($\ref{growth1}$).  In particular, we shall replace $(\ref{growth1})$ with the condition
\begin{align}
S'(Z)\leq -\frac{C_2}{Z\log(3+Z)},\label{growth2}
\end{align}
retaining the other constitutive assumptions established in Section $2$.  Such expanded growth conditions are relevant in a variety of physical models; see for instance \cite{FeireislNovotnyProcRSE135A_2005} for a typical example.  For technical reasons, it is necessary in this case to assume that the initial temperature $\theta_0$ is bounded away from zero.  
\begin{proposition}
\label{thm_temp3}
Let $\Omega$ be a bounded open set, and let $T>0$ be given.  Suppose that $\mathbb{S}$, $\kappa$, $\sigma$, and $s$, $p$ satisfy the criteria established in Section $\ref{prelim}$ with $(\ref{growth1})$ replaced by $(\ref{growth2})$ for some $C_2>0$.  

Let $(\rho,u,\theta)$ be weak solution of the Navier-Stokes-Fourier system $(\ref{nsf_system})$ satisfying the local entropy inequality ($\ref{eq_admiss}$) for a.e. $0<t<T$ with $s=0$, together with the bounds $\rho\in L^\infty([0,T]\times\Omega)$, $u\in L^2(0,T;H^1_0(\Omega))$ and
\begin{align}
\exists \quad\overline{\theta}>0\quad\textrm{such that}\quad \theta_0\geq \overline{\theta}\quad\textrm{for a.e.}\quad x\in \Omega.\label{thm3-initial}
\end{align}

Then there exists $\eta_{T}>0$ such that 
\begin{align*}
\theta(t,x)\geq \eta_{T}.
\end{align*}
for a.e. $0\leq t<T$, and a.e. $x\in\Omega$.
\end{proposition}
The proof is largely similar to the proof of Proposition $\ref{thm_temp1}$ given above, with some slight adjustment to account for the different assumption on the intial temperature profile $x\mapsto \theta(0,x)$.

\begin{proof}
In this setting, we again let $(C_k)\subset\mathbb{R}^+$ be a decreasing sequence, and define
\begin{align*}
V_k&:=\esssup_{0\leq t\leq T}\int_{\Omega} \rho W(\theta,C_{k},\lVert \rho\rVert_{L^\infty})dx\\
&\hspace{0.2in}+\int_0^T\int_{\Omega} \frac{\eta}{\theta}|\ebdiv u|^2\chi_{\theta\leq C_k}(t,x)dxdt\\
&\hspace{0.2in}+\int_0^T\int_{\Omega} \frac{\kappa |\nabla\theta|^2}{\theta^2}\chi_{\theta\leq C_k}(t,x)dxdt,\\
W(a,b,c)&:=\chi_{a<b}(a,b)\int_a^b \frac{1}{\omega\log(3+\frac{c}{\omega^{3/2}})}d\omega.
\end{align*}
Accordingly, arguing as in Step $1$ of the proof of Theorem $\ref{thm_temp1}$, we use the inequality
\begin{align*}
s(\rho,C_{k+1})-s(\rho,\theta)&\geq C\int_{\theta}^{C_{k+1}} \frac{3}{2\omega \log(3+\frac{\rho}{\omega^{3/2}})}d\omega\\
&\geq CW(\theta,C_{k+1},\lVert \rho\rVert_{L^\infty}) 
\end{align*}
along with ($\ref{eq_admiss}$) to obtain
\begin{align*}
V_{k+1}\leq C^*
\end{align*}
for some $C^*>0$, as well as
\begin{align*}
V_{k+1}&\leq C\int_0^T\int_{\Omega} \chi_{\theta\leq C_{k+1}}\rho^2(-\partial_\rho s(\rho,C_{k+1}))|\ebdiv u|dxdt\\
&\leq \left\lVert \frac{\eta^{1/2}\ebdiv u}{\theta^{1/2}}\chi_{\theta\leq C_k}\right\rVert_{L^2([0,T];L^2(\Omega))}\lVert \rho\chi_{\theta\leq C_{k+1}}\rVert_{L^2([0,T];L^2(\Omega))}\\
&\leq CV_k^{1/2}(I^{(W)})^{1/2}
\end{align*}
provided that $C_0$ is sufficiently small, where we have set
\begin{align*}
(I^{(W)})&:=\int_0^T\int_{\Omega} \rho(t,x)^2\chi_{\theta\leq C_{k+1}}(t,x)dxdt.
\end{align*}
Note that in performing this calculation, we have used the hypothesis ($\ref{thm3-initial}$) to eliminate the term corresponding to the initial condition.  Now, setting
\begin{align*}
F^{(W)}_{k}(\theta)&:=W(\theta,C_k,\lVert \rho\rVert_{L^\infty}),\\
R^{(W)}_{k}&:=W(C_{k+1},C_k,\lVert \rho\rVert_{L^\infty}),
\end{align*}
a simple calculation shows that $R^{(W)}_k\leq F^{(W)}_k(\theta(t,x))$ holds on the set $\{\theta<C_{k+1}\}$.  From here, we may proceed as in the proof of Theorem $\ref{thm_temp1}$; we include the full details for completeness.  Fixing $\alpha, \beta, p$ and $q$ satisfying
\begin{align*}
\alpha\in (0,2),\quad \beta>0\quad\textrm{and}\quad p,q\geq 1
\end{align*}
we argue as in ($\ref{eqaa1}$) to obtain the estimate
\begin{align*}
(I^{(W)})&\leq \frac{1}{(R^{(W)}_k)^\beta}\int_0^T\int_{\Omega}\rho(t,x)^2F^{(W)}_k(\theta(t,x))^\beta dxdt\\
&\leq \frac{C}{(R^{(W)}_k)^\beta}\lVert \rho\rVert_{L^{\infty}([0,T];L^{\omega}(\Omega))}^{2-\alpha}\lVert \rho^\alpha F^{(W)}_k(\theta)^\beta\rVert_{L^{p'}([0,T];L^{q'}(\Omega))}
\end{align*}
provided that ($\ref{req2}$) holds.  Proceeding as in ($\ref{eqaa2}$)-($\ref{eq1bd}$) (and in particular using Lemma $\ref{lem4}$), we write
\begin{align*}
\lVert \rho^\alpha (F^{(W)}_k)^\beta\rVert_{L^{p'}L^{q'}}&\leq \lVert \rho F^{(W)}_k\rVert_{L^\infty L^1}^\alpha \lVert F^{(W)}_k\rVert_{L^2L^6}^{\beta-\alpha}\\
&\leq CV_k^\alpha(\lVert \rho F^{(W)}_k\rVert_{L^\infty L^1}+\lVert \nabla F^{(W)}_k\rVert_{L^2L^2})^{\beta-\alpha}\\
&\leq CV_k^{\alpha}\left(V_k+\left\lVert \chi_{\theta\leq C_k}\frac{\nabla \theta}{\theta}\right\rVert_{L^2L^2}\right)^{\beta-\alpha}
\end{align*}
whenever ($\ref{req3}$) is satisfied, where we have observed that $\log(3+\frac{\lVert \rho\rVert_{L^\infty}}{\theta^{3/2}})\geq 1$ for a.e. $(t,x)\in (0,T)\times \Omega$.  We therefore obtain
\begin{align*}
(I^{(W)})&\leq \frac{C}{(R^{(W)}_k)^\beta}V_k^\alpha(V_k+V_k^{1/2})^{\beta-\alpha}\\
&\leq \frac{C}{(R^{(W)}_k)^\beta}(V_k^\beta+V_k^{(\alpha+\beta)/2})
\end{align*}
as before, so that
\begin{align*}
V_{k+1}&\leq \frac{C}{(R^{(W)}_k)^{\beta/2}}(V_k^{(1+\beta)/2}+V_k^{(2+\alpha+\beta)/4}).
\end{align*}
Now, fixing $M>0$ and choosing $\alpha, \beta, C_k$ and $T_k$ as in Step $5$ of Theorem $\ref{thm_temp1}$, we obtain
\begin{align*}
\frac{V_{k+1}}{C^*}\leq \frac{C^{k+1}}{C^*M^{\beta/2}}V_k^{\gamma_1}\leq C^{k+1}\left(\frac{V_k}{C^*}\right)^{\gamma_1}
\end{align*}
for some $\gamma_1>1$.  By Lemma $\ref{lem_ineq}$, we may therefore choose $C_0^*>0$ such that $V_1\leq C_0^*$ implies $V_k\rightarrow 0$ as $k\rightarrow\infty$.  As before, we may choose $M$ large enough so that this smallness condition for $V_1$ is satisfied.  We then obtain $\log(e^{-M}/f_{e^{-M}}(\theta))$ constant in $x$ for a.e. $t\in [0,T]$; this implies that for a.e. $t\in [0,T]$, either $\theta(t,x)\geq e^{-M}$ for a.e. $x\in \Omega$ (in which case the proof is complete), or 
\begin{align}
x\mapsto \theta(t,x)\quad\textrm{is constant a.e. on}\quad \Omega.\label{eqAAA115}
\end{align}
Suppose now that ($\ref{eqAAA115}$) holds.  Then, letting $\theta(t)=\theta(t,0)$, we obtain
\begin{align*}
0&=\int_\Omega \rho(t,x)W(\theta,e^{-M},\lVert \rho\rVert_{L^\infty})dx\\
&=W(\theta(t),e^{-M},\lVert \rho\rVert_{L^\infty})\int_{\Omega}\rho(t,x)dx\\
&=W(\theta(t),e^{-M},\lVert \rho\rVert_{L^\infty})\int_{\Omega}\rho_0dx
\end{align*}
for a.e. $t\in [0,T]$, where we have again used the conservation of mass as in the proof of Proposition $\ref{thm_temp1}$.  We therefore have $W(\theta(t,x),e^{-M},\lVert \rho\rVert_{L^\infty})=0$, and thus $\theta\geq e^{-M}$ for a.e. $t$ and $x$ as desired.  This completes the proof of Proposition $\ref{thm_temp3}$.
\end{proof}


\begin{thebibliography}{2}
\bibitem{BeiraoDaVeiga} Beir\~ao da Veiga, H. Concerning the regularity of the solutions to the Navier-Stokes equations via the truncation method, I.  Diff. Int. Eq. 10 (1997), 1149--1156.
\bibitem{CKN} Caffarelli, L., Kohn, R. and Nirenberg, L.  Partial regularity of suitable weak solutions of the Navier-Stokes equations.  Comm. Pure Appl. Math. 35 (1982), 771--831.
\bibitem{CVAnnals} Caffarelli, L. and Vasseur, A. Drift diffusion equations with fractional diffusion and the quasi-geostrophic equation. Ann. of Math. (2) 171 (2010), no. 3, 1903--1930.
\bibitem{CaffarelliVasseur} Caffarelli, L. and Vasseur, A. The De Giorgi method for regularity of solutions of elliptic equations and its applications to fluid dynamics. Discr. Contin. Dyn. Syst. Ser. S. 3 (2010), no. 3, 409--427.
\bibitem{Chan} Chan, Ch.H. Smoothness criterion for Navier-Stokes equations in terms of regularity along the streamlines. Methods Appl. Anal. 17 (2010), no. 1, 81--103.
\bibitem{DeGiorgi} De Giorgi, E. Sulla differenziabilità e l’analiticit\'a delle estremali degli integrali multipli regolari. Mem. Accad. Sci. Torino. Cl. Sci. Fis. Mat. Nat. 3 (1957), 25--43.
\bibitem{DiPerna-Lions} DiPerna, R.J. and Lions, P.-L. Ordinary differential equations, transport theory and Sobolev spaces. Invent. Math. 98 (1989), no. 3, 511--547.
\bibitem{FeireislBook} Feireisl, E. Dynamics of viscous compressible fluids. Oxford Lect. Series in Mathematics and its Applications, 26. Oxford Univ. Press, Oxford (2004).
\bibitem{FeireislIUMJ53_2004} Feireisl, E. On the motion of a viscous, compressible, and head conducting fluid. Indiana Univ. Math. J. 53 (2004), no. 6, 1705--1738.
\bibitem{FeireislContempMath371} Feireisl, E. Mathematics of viscous, compressible, and heat conducting fluids.  Nonlinear partial differential equations and related analysis, Contemp. Math. 371, Amer. Math. Soc., Providence, RI (2005), 133--151.
\bibitem{FeireislCommPDE31_2006} Feireisl, E. Stability of flows of real monoatomic gases. Comm. PDE 31 (2006), no. 1-3, 325--348.
\bibitem{FeireislCMAP} Feireisl, E. Mathematical theory of compressible, viscous, and heat conducting fluids.  Comp. Math. Appl. 53 (2007), 461--490.
\bibitem{FeireislNovotnyProcRSE135A_2005} Feireisl, E. and Novotn\'y, A.  On a simple model of reacting compressible flows arising in astrophysics.  Proc. Roy. Soc. Edinburgh Sect. A 135 (2005), no. 6, 1169--1194.
\bibitem{FeireislNovotnyBook} Feireisl, E. and Novotn\'y, A.  Singular limits in thermodynamics of viscous fluids.  Advances in Mathematical Fluid Mechanics.  Birkh\"auser Verlag, Basel (2009).
\bibitem{FNsuitable} Feireisl, E., Novotn\'y, A. and Sun, Y. Suitable weak solutions to the Navier-Stokes equations of compressible viscous fluids.  Indiana Univ. Math. J. (2012)
\bibitem{MelletVasseur} Mellet, A., Vasseur A.  A bound from below for the temperature in compressible Navier-Stokes equations.  Monatsh. Math. 157 (2009), no. 2, 143--161.
\bibitem{Lions} Lions, P.-L. Mathematical Topics in Fluid Mechanics, Volume 2: Compressible Models, Oxford Lectures Series in Mathematics and its Applications, 10.  Oxford Univ. Press, Oxford (1998).
\bibitem{VasseurNoDEA} Vasseur, A. A new proof of partial regularity of solutions to Navier-Stokes equations.  Nonlinear Diff. Eq. Appl. 14 (2007), no. 5-6, 753--785.
\end{thebibliography}
\end{document}